\documentclass[11pt,a4paper]{article}
\usepackage{amsmath,amssymb,amsthm,mathrsfs,fullpage,lineno,bbm}
\usepackage{hyperref}

\newtheorem{theorem}{Theorem}
\newtheorem{lemma}{Lemma}
\newtheorem{proposition}{Proposition}
\newtheorem{corollary}{Corollary}
\theoremstyle{definition}
\newtheorem{definition}{Definition}
\newtheorem{conjecture}{Conjecture}
\newtheorem{remark}{Remark}

\newtheorem{fact}{Fact}
\newtheorem{claim}{Claim}
\newcommand{\E}{\mathbb{E}}
\newcommand{\Prb}{\mathbb{P}}

\newcommand{\1}{\mathbf{1}}
\newcommand{\Inf}{\mathrm{Inf}}

\newcommand{\Icorr}{\mathcal{I}}
\newcommand{\Cov}{\mathrm{Cov}}

\newcommand\tup[1]{\left\langle #1 \right\rangle}
\newcommand{\cube}{\{0,1\}}

\title{A Spectral Correlation Inequality for Increasing Boolean Functions}

\begin{document}
\author{Fan Chang\thanks{School of Statistics and Data Science, Nankai University, Tianjin, China; and Extremal Combinatorics and Probability Group, Institute for Basic Science, Daejeon, South Korea. Email: \texttt{1120230060@mail.nankai.edu.cn}. Supported by the National Natural Science Foundation of China (NSFC) under grant 124B2019 and by the Institute for Basic Science (IBS-R029-C4).}
}
\date{}

\maketitle

\begin{abstract}
Talagrand's correlation inequality provides a quantitative strengthening of the Harris--Kleitman inequality for increasing Boolean functions. Motivated by a Fourier-analytic conjecture of Friedgut, Kahn, Kalai, and Keller~\cite[Conjecture 5.8]{FKKK2018correlation}, we prove that
\[
\Cov(f,g)\ge
2\sum_{S\neq\emptyset}|S|\hat f(S)^2\hat g(S)^2
\]
holds for all increasing Boolean functions $f,g:\cube^n\to\cube$. The proof combines the reverse Bonami--Beckner inequality with Young's
convolution inequality. We also establish a sharp pointwise inequality: for every $n\ge1$, every $0\le\rho\le1$, and every $f,g:\cube^n\to[0,1]$, the optimal constant $c_{\rho,n}$ for which
\[
\tup{f,T_\rho g}\ge c_{\rho,n}\|f*g\|_2^2
\]
holds for all such $f,g$ is $1$ for $0\le\rho\le1/2$, $(2(1-\rho))^n$ for $1/2<\rho<1$, and $0$ for $\rho=1$. Integrating this pointwise inequality yields, for $n\ge1$, the slightly improved bound
\[
\Cov(f,g)\ge 4\cdot\frac{n+1}{2n}
\sum_{S\neq\emptyset}|S|\hat f(S)^2\hat g(S)^2.
\]
\end{abstract}

\section{Introduction}\label{sec:intro}
We work on the discrete cube $\cube^n$ with the uniform product measure. We write $\E$ and $\Prb$ for expectation and probability, and
\[
\tup{f,g}:=\E[fg],
\quad
\Cov(f,g):=\E[fg]-\E[f]\E[g].
\]
The cube is ordered coordinatewise: $x\le y$ means $x_i\le y_i$ for all $i\in[n]$.

\begin{definition}
A function $f:\cube^n\to\mathbb R$ is \emph{increasing} if
\[
x\le y \quad\Longrightarrow\quad f(x)\le f(y).
\]
\end{definition}

A central theme in the analysis of Boolean functions is that monotonicity forces positive correlation. The classical Harris--Kleitman inequality~\cite{Harris1960,Kleitman1966} asserts that if $f,g:\cube^n\to\cube$ are increasing, then $\Cov(f,g)\ge 0$. A more delicate question is to quantify how much positive correlation is forced by the coordinates on which $f$ and $g$ simultaneously depend. For a Boolean function $f:\cube^n\to\cube$, the influence of the $i$-th coordinate is
\[
\Inf_i[f]:=\Prb\left[f(x)\neq f(x\oplus e_i)\right],
\]
where $e_i$ is the $i$-th standard basis vector and $\oplus$ denotes addition modulo $2$. Following Talagrand, we define the cross-total-influence of $f$ and $g$ by
\[
\Icorr(f,g):=\sum_{i=1}^n\Inf_i[f]\Inf_i[g].
\]
Talagrand~\cite{Talagrand1996correlated} gives the following logarithmic lower bound.
\begin{theorem}[Talagrand's correlation inequality~\cite{Talagrand1996correlated}]\label{thm:Talagrand correlation}
There exists a universal constant $c>0$ such that for every pair of increasing Boolean functions $f,g:\cube^n\to\cube$,
\begin{equation}\label{ineq:Talagrand correlation}
\Cov(f,g)
\ge
c\,\varphi(\Icorr(f,g)),
\qquad
\varphi(x)=\frac{x}{\log(e/x)}.
\end{equation}
\end{theorem}
The logarithmic correction is known to be necessary in general. Tight examples include small Hamming balls and their duals~\cite{Talagrand1996correlated}, Tribes-type examples~\cite{Keller09correlation}, and halfspaces and their duals~\cite[Corollary~1.2]{KK2019DA}; see also~\cite{KKM2016correlation} for structural conditions related to tightness.

Friedgut, Kahn, Kalai, and Keller~\cite[Conjecture 5.8]{FKKK2018correlation} proposed the following spectral strengthening of Harris--Kleitman.
\begin{conjecture}[Friedgut--Kahn--Kalai--Keller~\cite{FKKK2018correlation}]\label{conj:spectral-FKKK}
For any increasing Boolean functions $f,g:\cube^n\to\cube$,
\begin{equation}\label{eq:spectral-conj}
 \Cov(f,g)\ge 4\sum_{S\neq\emptyset}|S|\hat f(S)^2\hat g(S)^2.
\end{equation}
\end{conjecture}

\begin{remark}\label{rem:FKKK-motivation}
Conjecture~\ref{conj:spectral-FKKK} is motivated by the correlation approach to Chv\'atal's conjecture. Chv\'atal's conjecture~\cite{Chvatal1974} asks whether every decreasing family has a largest intersecting subfamily that is a star. Friedgut--Kahn--Kalai--Keller~\cite{FKKK2018correlation} reformulated this problem in terms of influences and correlation inequalities, reducing one of its central forms to a lower bound on $\Cov(f,g)$ when $f$ is increasing and $g$ is both increasing and antipodal. They then proposed several strengthenings of Harris--Kleitman inequality in Fourier language. Conjecture~\ref{conj:spectral-FKKK} is a diagonal spectral form of this program: it asks the covariance to dominate the common Fourier mass of $f$ and $g$, with each level weighted by its degree. Proposition~\ref{prop:or-majority} in Section~\ref{sec:remarks} shows that the stronger ``dream relation'' from~\cite[Section~2.3]{FKKK2018correlation}, and even any universal-constant version of that relation, fails when the second function is increasing and antipodal: in that example $\Cov(f,g)\asymp 2^{-n}$ while $\sum_i\Inf_i[f]\Inf_i[g]\asymp \sqrt n\,2^{-n}$. Thus the example rules out that particular route from~\cite[Section~2.3]{FKKK2018correlation} toward proving Chv\'atal's conjecture.
\end{remark}

Recently, Chang and Chen~\cite{ChangChen2025submodular} verified Conjecture~\ref{conj:spectral-FKKK} in the structured setting where the two increasing Boolean functions are either both submodular or both supermodular. We prove the following lower bound, which is weaker than Conjecture~\ref{conj:spectral-FKKK} only by a factor of $2$ in the constant.
\begin{theorem}\label{thm:spectral-lower-bound}
For all increasing Boolean functions $f,g:\cube^n\to\cube$,
\begin{equation}\label{eq:half-bound-intro}
\Cov(f,g)\ge
2\sum_{S\neq\emptyset}|S|\hat f(S)^2\hat g(S)^2.
\end{equation}
\end{theorem}

\begin{remark}\label{rem:FKKK-corollary-56}
Theorem~\ref{thm:spectral-lower-bound} implies Corollary~5.6 of~\cite{FKKK2018correlation}, which asserts that for increasing Boolean functions $f$ and $g$, with $g$ balanced, there is a universal $c>0$ such that
\[
\Cov(f,g)\ge c\sum_{S\neq\emptyset}\hat{f}(S)^2\hat{g}(S)^2.
\]
After that corollary, Friedgut--Kahn--Kalai--Keller wrote that ``It would be interesting to extend it to other contexts and to find a proof that's more direct than the one in Section~6.3.'' Theorem~\ref{thm:spectral-lower-bound} gives such a direct proof with an explicit constant, without the balance assumption, and with the stronger weight $|S|$.
\end{remark}

To approach Conjecture~\ref{conj:spectral-FKKK}, we use the covariance interpolation formula to reformulate Conjecture~\ref{conj:spectral-FKKK} as a semigroup-convolution inequality. For $0\le\rho\le1$, the Bonami--Beckner noise operator~\cite{Beckner1975,Bonami1970} $T_\rho$ acts on functions $f:\cube^n\to \mathbb{R}$ in the following way: $T_{\rho}f(x)$ is the average of $f(y)$ over $y\in\cube^n$ that are $\rho$-correlated with $x$, i.e. for every $i\in [n]$, $y_i= x_i$ with probability $\rho$, and otherwise $y_i$ is resampled uniformly and independently from $\cube$.

We also use the normalized convolution on the group $\mathbb F_2^n$:
\[
(f*g)(z):=\E_x[f(x)g(x\oplus z)],
\quad \forall z\in\cube^n.
\]
By the covariance interpolation formula~\eqref{eq:semigroup-derivative-KMS} and the Fourier spectral property of convolution in Fact~\ref{fact:convolution-fourier}, Conjecture~\ref{conj:spectral-FKKK} is equivalent to the following statement.
\begin{conjecture}\label{conj:semigroup-convolution}
For any increasing Boolean functions $f,g:\cube^n\to\cube$,
\begin{equation}\label{eq:semigroup-convolution-conj}
\sum_{i=1}^n\int_0^1\tup{\partial_i f,T_\rho\partial_i g}\,d\rho
\ge
\sum_{i=1}^n\|\partial_i f*\partial_i g\|_2^2,
\end{equation}
where
\[
\partial_i h(x):=h\left(x^{(i\to1)}\right)-h\left(x^{(i\to0)}\right),
\quad x\in\cube^{[n]\setminus\{i\}},
\]
and $x^{(i\to b)}$ denotes the vector obtained from $x\in\cube^{[n]\setminus\{i\}}$ by inserting $b\in\{0,1\}$ in the $i$-th coordinate.
\end{conjecture}

The proof of Theorem~\ref{thm:spectral-lower-bound} is based on a weak averaged form of Conjecture~\ref{conj:semigroup-convolution}. It follows from the standard reverse hypercontractivity together with Young's convolution inequality.
\begin{theorem}\label{thm:averaged-rbb-young}
Let $n\ge0$ and let $f,g:\cube^n\to[0,1]$. Then
\begin{equation}\label{eq:averaged-rbb-young-intro}
\int_0^1\tup{f,T_\rho g}\,d\rho
\ge
\frac12\|f*g\|_2^2.
\end{equation}
\end{theorem}
Applying Theorem~\ref{thm:averaged-rbb-young} to $\partial_i f,\partial_i g$ and summing over $i$ gives Theorem~\ref{thm:spectral-lower-bound}; see Section~\ref{sec:rbb-young}.

We also prove a sharp pointwise theorem between noise correlation and convolution energy for arbitrary $[0,1]$-valued functions,
identifying the optimal constant for each fixed $0\le \rho\le 1$.
\begin{theorem}\label{thm:general-noise-comparison}
Let $n\ge0$ and let $0\le\rho\le1$. Let $c_{\rho,n}$ be the largest constant such that, for every $f,g:\cube^n\to[0,1]$,
\[
    \tup{f,T_\rho g}\ge c_{\rho,n}\|f*g\|_2^2.
\]
If $n=0$, then $c_{\rho,0}=1$ for every $0\le\rho\le1$. If $n\ge1$, then
\[
    c_{\rho,n}=
    \begin{cases}
        1, & 0\le \rho\le \frac12,\\[4pt]
        \bigl(2(1-\rho)\bigr)^n, & \frac12<\rho<1,\\[4pt]
        0, & \rho=1.
    \end{cases}
\]
\end{theorem}
The proof proceeds by induction on the dimension, following the spirit of standard inductive proofs of hypercontractive inequalities. Integrating this pointwise theorem over $\rho$ gives the following slightly stronger spectral lower bound.
\begin{corollary}\label{cor:integrated-noise-bound}
For every $n\ge1$ and all increasing Boolean functions $f,g:\cube^n\to\cube$,
\[
\Cov(f,g)\geq 4\cdot\frac{n+1}{2n}
\sum_{S\neq\emptyset}|S|\hat f(S)^2\hat g(S)^2.
\]
\end{corollary}

\medskip
\noindent\emph{Organization.}
This paper is organized as follows. Section~\ref{sec:prelim} contains the preliminaries on Fourier analysis over the hypercube, discrete derivatives, the noise operator, and convolution notation used throughout the paper; it also proves the covariance interpolation formula and the formal equivalence between Conjectures~\ref{conj:spectral-FKKK} and~\ref{conj:semigroup-convolution}. Section~\ref{sec:rbb-young} proves Theorem~\ref{thm:averaged-rbb-young} and Theorem~\ref{thm:spectral-lower-bound}. Section~\ref{sec:sharp-comparison} proves Theorem~\ref{thm:general-noise-comparison} and Corollary~\ref{cor:integrated-noise-bound}. Section~\ref{sec:remarks} records the connection with Chv\'atal's conjecture and explains why the dream-relation approach from~\cite[Section~2.3]{FKKK2018correlation} cannot prove it.

\section{Preliminaries}\label{sec:prelim}
We consider real-valued functions $f:\{0,1\}^n \to \mathbb{R}$, equipped with the inner product $\tup{f,g}=\mathbb{E}_x[f(x)g(x)]$. As above, $\Cov(f,g)$ denotes covariance. For $S\subseteq[n]$, define the Fourier--Walsh character by $\chi_S(x):=(-1)^{\sum_{i\in S}x_i}$. The family $
\{\chi_S\}_{S\subseteq[n]}$ is an orthonormal basis of $L^2(\{0,1\}^n)$. The Fourier--Walsh expansion of $f$ is given by $f(x)=\sum_{S\subseteq[n]}\hat{f}(S)\chi_S(x)$, where $\hat{f}(S)=\tup{f,\chi_S}$. For $p>0$, we write $\|f\|_p:=\left(\mathbb{E}_x [|f(x)|^p]\right)^{1/p}$, with the expectation taken over the relevant cube. This is the usual $L^p$-norm for $p\ge1$, and the standard $L^p$ quasi-norm for $0<p<1$.

For $i\in[n]$, the $i$-th discrete derivative is the function on $\cube^{[n]\setminus\{i\}}$ given by
\[
\partial_i f(x):=f(x^{(i\to1)})-f(x^{(i\to0)}),
\]
where $x^{(i\to b)}$ denotes the point obtained from $x\in\cube^{[n]\setminus\{i\}}$ by inserting $b\in\{0,1\}$ in the $i$-th coordinate.

\begin{fact}\label{fact:fourier-derivative}
Let $f:\cube^n\to\mathbb R$ and $i\in[n]$. Then
\[
\partial_i f(x)=-2\sum_{S\ni i}\hat f(S)\chi_{S\setminus\{i\}}(x).
\]
Consequently, $\E[\partial_i f]=-2\hat f(\{i\})$.
\end{fact}

The noise operator $T_\rho$, $0\le\rho\le1$, is defined by
\[
T_\rho f(x)=\sum_{S\subseteq[n]}\rho^{|S|}\hat f(S)\chi_S(x).
\]
Equivalently, $T_\rho f(x)=\E[f(Y)]$, where each coordinate of $Y$ is kept equal to the corresponding coordinate of $x$ with probability $\rho$ and is otherwise resampled uniformly. Thus its Markov kernel is
\begin{equation}\label{eq:noise-kernel}
P_\rho(x,y)=2^{-n}(1+\rho)^{n-d(x,y)}(1-\rho)^{d(x,y)},
\end{equation}
where $d(x,y)$ is Hamming distance. The operator $T_\rho$ is self-adjoint, positive, and satisfies $T_1={\rm id}$ and $T_0f=\E[f]$.

For $f,g:\cube^n\to\mathbb R$, define the convolution
\[
(f*g)(z):=\E_x[f(x)g(x\oplus z)],\quad z\in\cube^n.
\]

\begin{fact}[Fourier transform of convolution; see {\cite[Definition~1.24 and Theorem~1.27]{ryanbook2014}}]\label{fact:convolution-fourier}
For any $f,g:\cube^n\to\mathbb R$ and any $S\subseteq[n]$,
\begin{equation}\label{eq:convolution-fourier}
\widehat{f*g}(S)=\hat f(S)\hat g(S).
\end{equation}
\end{fact}
We shall use the following standard semigroup representation for covariance, in a form closely related to that of Keller--Mossel--Sen~\cite{KMS2014correlation}.
\begin{lemma}\label{lem:semigroup-derivative-KMS}
For all $f,g:\cube^n\to\mathbb R$,
\begin{equation}\label{eq:semigroup-derivative-KMS}
\Cov(f,g)=\frac14\sum_{i=1}^n\int_0^1 \tup{\partial_i f,T_\rho\partial_i g}\,d\rho.
\end{equation}
\end{lemma}

\begin{proof}
For $0<\rho<1$, the Fourier representation of $T_\rho$ gives
\[
\frac{d}{d\rho}\tup{f,T_{\rho}g}
=\sum_{S\neq \emptyset}|S|\rho^{|S|-1}\hat f(S)\hat g(S).
\]
By Fact~\ref{fact:fourier-derivative},
\[
\frac14\sum_{i=1}^n\tup{\partial_i f,T_\rho\partial_i g}
=\sum_{i=1}^n\sum_{S\ni i}\rho^{|S|-1}\hat f(S)\hat g(S)
=\sum_{S\neq\emptyset}|S|\rho^{|S|-1}\hat f(S)\hat g(S).
\]
Therefore
\[
\frac{d}{d\rho}\tup{f,T_\rho g}
=\frac14\sum_{i=1}^n\tup{\partial_i f,T_\rho\partial_i g}.
\]
Integrating from $0$ to $1$ yields
\[
\Cov(f,g)=\tup{f,T_1g}-\tup{f,T_0g}=\frac14\sum_{i=1}^n\int_0^1\tup{\partial_i f,T_\rho\partial_i g}\,d\rho,
\]
since $\tup{f,T_1g}=\E[fg]$ and $\tup{f,T_0g}=\E[f]\E[g]$.
\end{proof}

We now prove the equivalence stated in the introduction.

\begin{lemma}\label{lem:conjecture-equivalence}
For increasing Boolean functions $f,g:\cube^n\to\cube$, inequality~\eqref{eq:spectral-conj} holds if and only if inequality~\eqref{eq:semigroup-convolution-conj} holds. Consequently, Conjecture~\ref{conj:spectral-FKKK} is equivalent to Conjecture~\ref{conj:semigroup-convolution}.
\end{lemma}

\begin{proof}
Combining Fact~\ref{fact:convolution-fourier} with Fact~\ref{fact:fourier-derivative} and summing over $i\in[n]$, for $f,g:\cube^n\to\mathbb R$ we have
\begin{equation}\label{eq:derivative-convolution-identity}
\frac14\sum_{i=1}^n\|\partial_i f*\partial_i g\|_2^2
=
4\sum_{S\neq\emptyset}|S|\hat f(S)^2\hat g(S)^2.
\end{equation}
Therefore by Lemma~\ref{lem:semigroup-derivative-KMS},
\[
\Cov(f,g)\ge 4\sum_{S\neq\emptyset}|S|\hat f(S)^2\hat g(S)^2
\quad\Longleftrightarrow\quad
\sum_{i=1}^n\int_0^1\tup{\partial_i f,T_\rho\partial_i g}\,d\rho
\ge
\sum_{i=1}^n\|\partial_i f*\partial_i g\|_2^2.\tag*{\qedhere}
\]
\end{proof}

\section{Proofs of Theorem~\ref{thm:spectral-lower-bound} and Theorem~\ref{thm:averaged-rbb-young} }\label{sec:rbb-young}

This section proves Theorem~\ref{thm:averaged-rbb-young} and Theorem~\ref{thm:spectral-lower-bound}. We first recall the two inequalities in the form needed below. 

\begin{theorem}[Reverse Two-Function Hypercontractivity
Theorem~{\cite[Exercise~10.6]{ryanbook2014}}]
\label{fact:reverse-BB}
Let $n\ge0$, let $0<p,q<1$, and let $0\le\rho\le\sqrt{(1-p)(1-q)}$. If $f,g:\cube^n\to[0,\infty)$, then
\[
\tup{f,T_\rho g}\ge \|f\|_p\|g\|_q.
\]
\end{theorem}

\begin{fact}[Young's convolution inequality~\cite{beltran2025optimal}]\label{fact:young-convolution}
Let $n\ge0$, and let $1\le p,q,r\le\infty$ satisfy $1+\frac1r=\frac1p+\frac1q$. Then, for all $f,g:\cube^n\to\mathbb R$,
\[
\|f*g\|_r\le \|f\|_p\|g\|_q.
\]
\end{fact}

\begin{proof}[Proof of Theorem~\ref{thm:averaged-rbb-young}]
If $\E[f]\E[g]=0$, then one of $f,g$ vanishes identically, and both sides of the desired inequality are zero. Hence assume $\E[f]\E[g]>0$.

For $0<\rho<1$, apply Theorem~\ref{fact:reverse-BB} with $p=q=1-\rho$. Since $0\le f,g\le1$, we have $f^{1-\rho}\ge f$ and $g^{1-\rho}\ge g$, and therefore
\[
\tup{f,T_\rho g}\ge \E[f]^{\frac{1}{1-\rho}}\cdot\E[g]^{\frac{1}{1-\rho}}.
\]
The endpoint $\rho=0$ follows by continuity, and the endpoint $\rho=1$ has measure zero and therefore does not affect the integral. Consequently,
\[
\int_0^1\tup{f,T_\rho g}\,d\rho
\ge
\int_0^1(\E[f]\E[g])^{\frac{1}{1-\rho}}\,d\rho.
\]
Write $\E[f]\E[g]=e^{-\lambda}$ with $\lambda\ge0$. With the change of variables $s=(1-\rho)^{-1}$,
\[
\int_0^1(\E[f]\E[g])^{\frac{1}{1-\rho}}\,d\rho
=
\int_1^\infty e^{-\lambda s}s^{-2}\,ds\ge
\int_1^\infty e^{-\lambda s}e^{-2(s-1)}\,ds
=
\frac{e^{-\lambda}}{\lambda+2},
\]
since $\log s\le s-1$ for $s\ge1$. Finally, $2e^{\frac{\lambda}{2}}\ge \lambda+2$ for $\lambda\ge0$, and so
\begin{equation}\label{eq:rbb-integrated-lower}
\int_0^1\tup{f,T_\rho g}\,d\rho
\ge
\frac12(\E f)^{\frac{3}{2}}(\E g)^{\frac{3}{2}}.
\end{equation}

On the other hand, Fact~\ref{fact:young-convolution}, with $p=q=\frac{4}{3}$ and $r=2$, gives
\[
\|f*g\|_2\le \|f\|_{4/3}\|g\|_{4/3}.
\]
Since $0\le f,g\le1$, $\|f\|_{4/3}^2=(\E f^{\frac{4}{3}})^{\frac{3}{2}}\le(\E f)^{\frac{3}{2}}$, and similarly $\|g\|_{4/3}^2\le(\E g)^{\frac{3}{2}}$. Therefore
\begin{equation}\label{eq:young-convolution-upper}
\|f*g\|_2^2\le(\E f)^{\frac{3}{2}}(\E g)^{\frac{3}{2}}.
\end{equation}
Combining~\eqref{eq:rbb-integrated-lower} and~\eqref{eq:young-convolution-upper} yields the desired inequality.
\end{proof}

We next derive Theorem~\ref{thm:spectral-lower-bound}.

\begin{proof}[Proof of Theorem~\ref{thm:spectral-lower-bound}]
Fix $i\in[n]$. Since $f$ and $g$ are increasing Boolean functions, $\partial_i f$ and $\partial_i g$ are $\cube$-valued functions on $\cube^{[n]\setminus\{i\}}$. Applying Theorem~\ref{thm:averaged-rbb-young} to $(\partial_i f,\partial_i g)$ gives
\[
\int_0^1\tup{\partial_i f,T_\rho\partial_i g}\,d\rho
\ge
\frac12\|\partial_i f*\partial_i g\|_2^2.
\]
Combining Lemma~\ref{lem:semigroup-derivative-KMS} with~\eqref{eq:derivative-convolution-identity} gives the desired bound.
\end{proof}

\section{Sharp pointwise theorem}\label{sec:sharp-comparison}

We now prove Theorem~\ref{thm:general-noise-comparison} in two steps. First we prove the half-noise inequality, which is the basic comparison at $\rho=1/2$. We then use tensorization and kernel comparison to obtain the sharp constant for every $\rho\in[0,1]$. After that, we integrate the theorem to prove Corollary~\ref{cor:integrated-noise-bound}.

\begin{lemma}[The half-noise inequality]\label{lem:half-noise}
Let $f,g:\cube^n\to[0,1]$. Then
\[
    \tup{f,T_{\frac{1}{2}}g}\ge \|f*g\|_2^2.
\]
\end{lemma}

\begin{proof}
We prove the statement by induction on $n$.

For $n=0$, the cube consists of a single point. Thus $f,g\in[0,1]$ are constants. Note that $\tup{f,T_{\frac{1}{2}}g}=fg$ and $\|f*g\|_2^2=f^2g^2$. Since $0\le f,g\le1$, we have $fg\ge f^2g^2$. This proves the base case.

Assume now that the result is known in dimension $n-1$. For any $x\in\{0,1\}^n$, write $x=(x',z)$ where $x'\in\{0,1\}^{n-1}$ and $z\in\{0,1\}$. Decompose $f$ and $g$ with respect to the last coordinate:
\[
f(x',z)=a(x')+ (-1)^z\cdot b(x'),\quad g(x',z)=c(x')+(-1)^z\cdot d(x'),
\]
where $a,b,c,d:\{0,1\}^{n-1}\to\mathbb R$ satisfy
$
a(x')=\frac{f(x',0)+f(x',1)}2, b(x')=\frac{f(x',0)-f(x',1)}2,c(x')=\frac{g(x',0)+g(x',1)}2,d(x')=\frac{g(x',0)-g(x',1)}2.
$
Hence $a+b=f(x',0)$, $a-b=f(x',1)$ and similarly $c+d=g(x',0)$, $c-d=g(x',1)$. Therefore, for every $\varepsilon,\eta\in\{\pm1\}$, the functions $a+\varepsilon b$ and $c+\eta d$ take values in $[0,1]$, since they are restrictions of $f$ and $g$.

Choose random signs $\varepsilon,\eta\in\{\pm1\}$ such that $\E[\varepsilon]=0$, $\E[\eta]=0$, and $\E[\varepsilon\eta]=\frac12$; for instance, take
\[
\Prb(\varepsilon=\eta=1)=\Prb(\varepsilon=\eta=-1)=\frac38,
\qquad
\Prb(\varepsilon=1,\eta=-1)=\Prb(\varepsilon=-1,\eta=1)=\frac18.
\]
By the induction hypothesis, for every fixed choice of $(\varepsilon,\eta)$,
\[
\tup{a+\varepsilon b,T_{\frac{1}{2}}(c+\eta d)}\ge\left\|(a+\varepsilon b)*(c+\eta d)\right\|_2^2.
\]
Averaging over $(\varepsilon,\eta)$, we obtain
\begin{equation}\label{eq:induction average condition}
\E_{\varepsilon,\eta}\left[\tup{a+\varepsilon b,T_{\frac{1}{2}}(c+\eta d)}\right]\ge\E_{\varepsilon,\eta}\left[\left\|(a+\varepsilon b)*(c+\eta d)\right\|_2^2\right].
\end{equation}

We now compute both sides of~\eqref{eq:induction average condition}.

First,
\[
\begin{aligned}
\E_{\varepsilon,\eta}\left[\tup{a+\varepsilon b,T_{\frac{1}{2}}(c+\eta d)}\right]
&=\E_{\varepsilon,\eta}\left[\tup{a,T_{\frac{1}{2}}c}+\eta\tup{a,T_{\frac{1}{2}}d}+\varepsilon\tup{b,T_{\frac{1}{2}}c}+\varepsilon\eta\tup{b,T_{\frac{1}{2}}d}\right]\\
&=\tup{a,T_{\frac{1}{2}}c}+\frac12\tup{b,T_{\frac{1}{2}}d}.
\end{aligned}
\]
\begin{claim}
$\tup{a,T_{\frac{1}{2}}c}+\frac12\tup{b,T_{\frac{1}{2}}d}=\tup{f,T_{\frac{1}{2}}g}$.
\end{claim}
Indeed, in the last coordinate,
$T_{\frac{1}{2}}$ multiplies the character $(-1)^z$ by $\frac{1}{2}$. Hence
$$
T_{\frac{1}{2}}g(x',z)=T_{\frac{1}{2}}c(x')+\frac{1}{2}(-1)^z T_{\frac{1}{2}}d(x'),
$$
where on the right-hand side $T_{\frac{1}{2}}$ is the noise operator on $\{0,1\}^{n-1}$. Therefore
\[
\begin{aligned}
\tup{f,T_{\frac{1}{2}}g}&=\tup{a+(-1)^zb,T_{\frac{1}{2}}c+\frac{1}{2}(-1)^z T_{\frac{1}{2}}d}=\tup{a+(-1)^zb,T_{\frac{1}{2}}c}+\frac{1}{2}\tup{b,T_{\frac{1}{2}}d}\\
&=\tup{a,T_{\frac{1}{2}}c}+\frac12\tup{b,T_{\frac{1}{2}}d},
\end{aligned}
\]
because the mixed terms containing a single factor $(-1)^z$ vanish after averaging over the last coordinate. Hence,
\begin{equation}\label{eq:LHSremaining}
\E_{\varepsilon,\eta}\left[\tup{a+\varepsilon b,T_{\frac{1}{2}}(c+\eta d)}\right]= \tup{f,T_{\frac{1}{2}}g}.  
\end{equation}

Next we compute the right-hand side of~\eqref{eq:induction average condition}. Fix $R\subseteq[n-1]$ and note that
$$
\widehat{a+\varepsilon b}(R)=\hat a(R)+\varepsilon\hat b(R), \quad \widehat{c+\eta d}(R)=\hat c(R)+\eta\hat d(R).
$$
Therefore
\[
\begin{aligned}
&\E_{\varepsilon,\eta}\left[\left\|(a+\varepsilon b)*(c+\eta d)\right\|_2^2\right]=\E_{\varepsilon,\eta}\left[\sum_{R\subseteq[n-1]}\widehat{a+\varepsilon b}(R)^2\widehat{c+\eta d}(R)^2\right]\\
&=\sum_{R\subseteq[n-1]}\E_{\varepsilon,\eta}\left[(\hat a(R)+\varepsilon\hat b(R))^2\cdot(\hat c(R)+\eta\hat d(R))^2\right]\\
&=\sum_{R\subseteq[n-1]}\left[\hat a(R)^2\hat c(R)^2+\hat a(R)^2\hat d(R)^2+\hat b(R)^2\hat c(R)^2+\hat b(R)^2\hat d(R)^2+2\hat a(R)\hat b(R)\hat c(R)\hat d(R)\right]\\
&=\sum_{R\subseteq[n-1]}\left[\hat a(R)^2\hat c(R)^2+\hat b(R)^2\hat d(R)^2+(\hat a(R)\hat d(R)+\hat b(R)\hat c(R))^2\right]\\
&\ge\sum_{R\subseteq[n-1]}\left[\hat a(R)^2\hat c(R)^2+\hat b(R)^2\hat d(R)^2\right].
\end{aligned}
\]
It remains to identify the last sum. Since $f=a+(-1)^zb$ and $g=c+(-1)^zd$, the Fourier coefficients of $f$ and $g$ are given by $\hat f(R)=\hat a(R),\hat g(R)=\hat c(R)$ for $R\subseteq[n-1]$, and $\hat f(R\cup\{n\})=\hat b(R),\hat g(R\cup\{n\})=\hat d(R)$. Consequently,
$$
\|f*g\|_2^2=\sum_{S\subseteq[n]}\hat f(S)^2\hat g(S)^2=\sum_{R\subseteq[n-1]}\hat a(R)^2\hat c(R)^2+\sum_{R\subseteq[n-1]}\hat b(R)^2\hat d(R)^2.
$$
Thus
\begin{equation}\label{eq:RHSremaining}
\E_{\varepsilon,\eta}\left[\left\|(a+\varepsilon b)*(c+\eta d)\right\|_2^2\right]\ge\|f*g\|_2^2.
\end{equation}

Combining~\eqref{eq:induction average condition}, \eqref{eq:LHSremaining},  and~\eqref{eq:RHSremaining}, we obtain
\[
\begin{aligned}
\tup{f,T_{\frac{1}{2}}g}
&=\E_{\varepsilon,\eta}\left[\tup{a+\varepsilon b,T_{\frac{1}{2}}(c+\eta d)}\right]\\
&\ge\E_{\varepsilon,\eta}\left[\left\|(a+\varepsilon b)*(c+\eta d)\right\|_2^2\right]\\
&\ge\|f*g\|_2^2.
\end{aligned}
\]
This completes the induction and proves the lemma.
\end{proof}

\begin{proof}[Proof of Theorem~\ref{thm:general-noise-comparison}]
If $n=0$, the cube has one point and $T_\rho$ is the identity for every $\rho$. Thus $\tup{f,T_\rho g}=fg$ and $\|f*g\|_2^2=f^2g^2$ for constants $f,g\in[0,1]$. The largest admissible constant is therefore $1$, with equality attained at $f=g=1$. We now assume $n\ge1$.

We first prove the lower bounds. The case $0\le \rho\le\frac12$ follows from Lemma~\ref{lem:half-noise} by an averaging argument.

For $\eta=(\eta_1,\ldots,\eta_n)\in[0,1]^n$, let $T_\eta$ denote the product noise operator
\[
    T_\eta=T^1_{\eta_1}T^2_{\eta_2}\cdots T^n_{\eta_n},
\]
where $T^i_{\eta_i}$ applies the one-coordinate noise operator in the $i$-th coordinate.

\begin{claim}\label{claim:anisotropic}
Let $J\subseteq[n]$, and define
\[
    \eta_i=
    \begin{cases}
        \frac12, & i\in J,\\
        0, & i\notin J.
    \end{cases}
\]
Then
\[
    \tup{f,T_\eta g}\ge \|f*g\|_2^2.
\]
\end{claim}
For each $z\in\{0,1\}^{J^c}$, define the restrictions
$$
f_{J^c\to z}(y)=f(x_J=y,x_{J^c}=z),\quad g_{J^c\to z}(y)=g(x_J=y,x_{J^c}=z).
$$
Here and below, convolution on sections is taken inside the cube $\{0,1\}^J$.

Since $T_\eta$ applies $T_{\frac12}$ in the coordinates of $J$ and complete resampling in the coordinates of $J^c$, we have
$$
\tup{f,T_\eta g}=\E_{z,z'\in\{0,1\}^{J^c}}\left[\tup{f_{J^c\to z},T_{\frac12}g_{J^c\to z'}}_{J}\right],
$$
where $\tup{\cdot,\cdot}_{J}$ denotes the normalized inner product on $\{0,1\}^J$.

By Lemma~\ref{lem:half-noise}, for every fixed $z,z'\in\{0,1\}^{J^c}$,
\[
\tup{f_{J^c\to z},T_{\frac12}g_{J^c\to z'}}_{J}\ge\|f_{J^c\to z}*g_{J^c\to z'}\|_{2,J}^2.
\]
Averaging over $z,z'$ gives
\begin{equation}\label{eq:anisotropic-lower}
\tup{f,T_\eta g}\ge\E_{z,z'\in\{0,1\}^{J^c}}\left[\|f_{J^c\to z}*g_{J^c\to z'}\|_{2,J}^2\right].
\end{equation}

On the other hand, write a point of the full cube as $(y,\delta)\in\{0,1\}^J\times\{0,1\}^{J^c}$, where $\delta$ is the $J^c$-coordinate of the convolution shift. Then
\[
(f*g)(y,\delta)
=
\E_{z\in\{0,1\}^{J^c}}
\left[(f_{J^c\to z}*g_{J^c\to z\oplus\delta})(y)\right].
\]
By Jensen's inequality,
\[
\begin{aligned}
\|f*g\|_2^2&=\E_{y\in\{0,1\}^{J},\delta\in\{0,1\}^{J^c}}\left[
\left(
\E_{z\in\{0,1\}^{J^c}}
\left[(f_{J^c\to z}*g_{J^c\to z\oplus\delta})(y)\right]
\right)^2
\right]\\
&\le
\E_{y\in\{0,1\}^{J},\delta\in\{0,1\}^{J^c},z\in\{0,1\}^{J^c}}
\left[
\left((f_{J^c\to z}*g_{J^c\to z\oplus\delta})(y)\right)^2
\right]\\
&=
\E_{z,z''\in\{0,1\}^{J^c}}
\left[\|f_{J^c\to z}*g_{J^c\to z''}\|_{2,J}^2\right],
\end{aligned}
\]
where $z''=z\oplus\delta$ and the addition $\oplus$ is coordinatewise addition modulo $2$. Combining this with~\eqref{eq:anisotropic-lower} proves Claim~\ref{claim:anisotropic}.

We now prove the case $0\le \rho\le\frac12$. In one coordinate,
\[
    T_\rho=(1-2\rho)T_0+2\rho T_{\frac12}.
\]
Tensorizing this identity, $T_\rho$ is a convex combination of the operators $T_\eta$ appearing in
Claim~\ref{claim:anisotropic}. More explicitly,
\[
T_\rho=\E_J[T_{\eta^J}],
\]
where $J$ is a random subset of $[n]$ obtained by putting each coordinate in $J$ independently with probability $2\rho$, and
\[
\eta_i^J=
    \begin{cases}
        \frac12, & i\in J,\\
        0, & i\notin J.
    \end{cases}
\]
Therefore,
$$
\tup{f,T_\rho g}=\E_{J\sim\mu_{2\rho}([n])}\left[\tup{f,T_{\eta^J}g}\right]\ge \|f*g\|_2^2,
$$
where the last inequality follows from Claim~\ref{claim:anisotropic}. This proves $c_{\rho,n}\ge 1$ for $0\le \rho\le\frac12$.

Next assume that $\frac12<\rho<1$.
Let $K_\rho(x,y)$ be the Markov kernel of $T_\rho$, so that
\[
    T_\rho g(x)=\sum_{y\in\{0,1\}^n}K_\rho(x,y)g(y).
\]
For one coordinate,
\[
    K_\rho(x_i,y_i)=
    \begin{cases}
        \dfrac{1+\rho}{2}, & x_i=y_i,\\[6pt]
        \dfrac{1-\rho}{2}, & x_i\ne y_i.
    \end{cases}
\]
Thus, for one coordinate,
\[
    \frac{K_{\frac12}(x_i,y_i)}{K_\rho(x_i,y_i)}
    =
    \begin{cases}
        \dfrac{3}{2(1+\rho)}, & x_i=y_i,\\[8pt]
        \dfrac{1}{2(1-\rho)}, & x_i\ne y_i.
    \end{cases}
\]
Since $\rho\ge\frac12$, we have $\frac{3}{2(1+\rho)}\le \frac{1}{2(1-\rho)}$. Therefore, coordinatewise,
\[
    K_{\frac12}(x_i,y_i)\le \frac{1}{2(1-\rho)}K_\rho(x_i,y_i).
\]
Multiplying over all coordinates gives
\[
    K_{\frac12}(x,y)
    \le
    \bigl(2(1-\rho)\bigr)^{-n}K_\rho(x,y).
\]
Since $g\ge0$, it follows that
\[
    T_{\frac12}g(x)
    \le
    \bigl(2(1-\rho)\bigr)^{-n}T_\rho g(x).
\]
Since also $f\ge0$, we get
\[
    \tup{f,T_{\frac12}g}
    \le
    \bigl(2(1-\rho)\bigr)^{-n}\tup{f,T_\rho g}.
\]
Using Lemma~\ref{lem:half-noise}, we conclude that
\[
    \|f*g\|_2^2
    \le
    \bigl(2(1-\rho)\bigr)^{-n}\tup{f,T_\rho g}.
\]
This proves $c_{\rho,n}\ge \bigl(2(1-\rho)\bigr)^n$ for $\frac12<\rho<1$.

It remains to prove that these constants are sharp.

First suppose $0\le \rho\le\frac12$. Taking $f\equiv 1$ and $g\equiv 1$,
we have $T_\rho g\equiv 1$ and $f*g\equiv 1$. Thus
\[
\tup{f,T_\rho g}=1=\|f*g\|_2^2.
\]
Hence no constant larger than $1$ can hold uniformly, and so $c_{\rho,n}=1$ for $0\le \rho\le\frac12$.

Now suppose $\frac12<\rho<1$. Let $f=\1_{\{0^n\}}$ and $g=\1_{\{1^n\}}$. Then $f*g=2^{-n}\1_{\{1^n\}}$ and hence $\|f*g\|_2^2=2^{-3n}$. On the other hand,
\[
    T_\rho g(0^n)
    =
    \left(\frac{1-\rho}{2}\right)^n.
\]
Therefore 
$$
\tup{f,T_\rho g}=2^{-n}T_\rho g(0^n)=2^{-n}\left(\frac{1-\rho}{2}\right)^n=2^{-2n}(1-\rho)^n.
$$
Consequently,
\[
\frac{\tup{f,T_\rho g}}{\|f*g\|_2^2}=\frac{2^{-2n}(1-\rho)^n}{2^{-3n}}=\bigl(2(1-\rho)\bigr)^n.
\]
Thus no constant larger than $\bigl(2(1-\rho)\bigr)^n$ can hold uniformly. Hence $ c_{\rho,n}=\bigl(2(1-\rho)\bigr)^n$ for $\frac12<\rho<1$.

Finally, for $\rho=1$, the same example gives $\tup{f,T_1g}=0$, while
$\|f*g\|_2^2=2^{-3n}>0$. Thus no positive constant $c$ can satisfy $\tup{f,T_1 g}\ge c\|f*g\|_2^2$ for all $f,g:\{0,1\}^n\to[0,1]$. 
\end{proof}

We now derive the spectral consequence stated in the introduction, namely Corollary~\ref{cor:integrated-noise-bound}.

\begin{proof}[Proof of Corollary~\ref{cor:integrated-noise-bound}]
For the one-point cube we use the convention $c_{\rho,0}=1$ for $0\le\rho\le1$. Indeed, if $u,v\in[0,1]$ are constants, then
\[
\tup{u,T_\rho v}=uv\ge u^2v^2=\|u*v\|_2^2.
\]
Fix $i\in[n]$. Since $f$ and $g$ are increasing Boolean functions, the derivatives $\partial_i f$ and $\partial_i g$ are $\cube$-valued functions on $\cube^{[n]\setminus\{i\}}$. If $n=1$, the preceding one-point estimate applies to $(\partial_i f,\partial_i g)$; if $n\ge2$, Theorem~\ref{thm:general-noise-comparison} applies in dimension $n-1$. Hence, in all cases,
\[
\int_0^1\tup{\partial_i f,T_\rho\partial_i g}\,d\rho
\ge \left(\int_0^1 c_{\rho,n-1}\,d\rho\right)
\|\partial_i f*\partial_i g\|_2^2.
\]
Using the explicit value of $c_{\rho,n-1}$, together with the convention $c_{\rho,0}=1$ when $n=1$,
\[
\int_0^1 c_{\rho,n-1}\,d\rho
=
\int_0^{\frac12}1\,d\rho+
\int_{\frac12}^1\bigl(2(1-\rho)\bigr)^{n-1}\,d\rho
=
\frac12+\frac1{2n}
=
\frac{n+1}{2n}.
\]
Hence combining Lemma~\ref{lem:semigroup-derivative-KMS} and~\eqref{eq:derivative-convolution-identity} gives
\[
\Cov(f,g)=\frac14\sum_{i=1}^n\int_0^1\tup{\partial_i f,T_\rho\partial_i g}\,d\rho\ge \frac{n+1}{2n}\cdot\frac14\sum_{i=1}^n\|\partial_i f*\partial_i g\|_2^2=4\cdot\frac{n+1}{2n}
\sum_{S\neq\emptyset}|S|\hat f(S)^2\hat g(S)^2.\tag*{\qedhere}
\]
\end{proof}

\section{Concluding Remarks}\label{sec:remarks}

We conclude by observing that the Chv\'atal correlation formulation cannot be obtained from the stronger ``dream relation'' proposed in~\cite[Section~2.3]{FKKK2018correlation}. Chv\'atal's conjecture asserts that for every decreasing family $\mathcal F\subseteq 2^{[n]}$, some largest intersecting subfamily of $\mathcal F$ is a star, i.e., has the form $\{A\in\mathcal F:\, i\in A\}$ for some $i\in[n]$. Friedgut, Kahn, Kalai, and Keller~\cite[Conjecture~1.2]{FKKK2018correlation} showed that this conjecture is equivalent to the following correlation formulation.

\begin{conjecture}[Friedgut--Kahn--Kalai--Keller~\cite{FKKK2018correlation}]
\label{conj:FKKK-antipodal}
For every increasing $f,g:\{0,1\}^n\to\{0,1\}$, if $g$ is antipodal, i.e., $g(x)=1-g(1-x)$ for any $x\in\cube^n$, then
\begin{equation}\label{ineq:FKKKineq}
\Cov(f,g)\ge \frac14 \min_{i\in[n]} \Inf_i[f].
\end{equation}
\end{conjecture}

The dream-relation approach from~\cite[Section~2.3]{FKKK2018correlation} would apply if one could prove a substantially stronger lower bound of the form
\begin{equation}\label{ineq:dream ineq}
\Cov(f,g)\ge C\,\Icorr(f,g)
\end{equation}
with a universal constant $C>0$. If such an inequality held for all increasing $f$ and increasing antipodal $g$, then Harper's edge-isoperimetric inequality applied to the balanced Boolean function $g$ would give $\sum_i\Inf_i[g]\ge1$, and hence
\[
\Cov(f,g)\ge C\min_i\Inf_i[f]\sum_{i=1}^n\Inf_i[g]\ge C\min_i\Inf_i[f].
\]
This would imply a Chv\'atal-type correlation inequality up to a universal constant. The following proposition shows that this particular route is impossible, namely a dimension-free lower bound of the form~\eqref{ineq:dream ineq} cannot work.

\begin{proposition}\label{prop:or-majority}
There is no universal constant $C>0$ such that
\[
\Cov(f,g)\ge C\,\Icorr(f,g)
\]
holds for all increasing Boolean functions $f,g:\cube^n\to\cube$, even if $g$ is antipodal.
\end{proposition}
\begin{proof}
Let $n$ be odd and define
\[
f:=\mathbbm{1}_{\{\sum_{j=1}^n x_j\ge 1\}},\quad g:=\mathbbm{1}_{\{\sum_{j=1}^n x_j\ge \frac{n+1}{2}\}}.
\]
Then $f$ and $g$ are increasing. Moreover, $g$ is antipodal. Indeed, for every $x\in\{0,1\}^n$, we have $\sum_{i=1}^n(1-x_i)=n-\sum_{i=1}^n x_i$, and since $n$ is odd, exactly one of $\sum_i x_i$ and $\sum_i(1-x_i)$ is at least $(n+1)/2$.

We now compute the covariance. Since $g=1$ implies $f=1$, we have $fg=g$. Also, $\E[f]=1-2^{-n}$ and $\E[g]=1/2$. Therefore
\[
\Cov(f,g)=\E[g]-\E[f]\E[g]
=
\frac12-\left(1-2^{-n}\right)\frac12
=
2^{-n-1}.
\]

We next compute the influences. For $f$, a coordinate $i$ is pivotal exactly when all the other $n-1$ coordinates are equal to $0$. Thus $\Inf_i[f]=2^{-(n-1)}$. For $g$, a coordinate $i$ is pivotal exactly when the other $n-1$ coordinates contain precisely $(n-1)/2$ ones. Hence
\[
\Inf_i[g]
=
\Prb\left(\sum_{j\neq i}x_j=\frac{n-1}{2}\right)
=
\frac{\binom{n-1}{(n-1)/2}}{2^{n-1}}.
\]
It follows that
\[
\Icorr(f,g)
=
\sum_{i=1}^n \Inf_i[f]\Inf_i[g]
=
 n\cdot \frac{1}{2^{n-1}}\cdot \frac{\binom{n-1}{(n-1)/2}}{2^{n-1}}
\asymp
\frac{\sqrt n}{2^n},
\]
by Stirling's formula. Consequently,
\[
\frac{\Cov(f,g)}{\Icorr(f,g)}
\asymp
\frac{2^{-n}}{2^{-n}\sqrt n}
=
\frac1{\sqrt n}
\longrightarrow 0
\]
as $n\to\infty$ through odd integers. Therefore no universal constant $C>0$ can make the dream relation valid in this antipodal setting. 
\end{proof}
This obstruction does not contradict Conjecture~\ref{conj:FKKK-antipodal}. In fact the present example attains equality there, since
\[
\Cov(f,g)=2^{-n-1}=\frac14\min_i\Inf_i[f].
\]

\bibliographystyle{plain}
\bibliography{reference}
\end{document}